\documentclass[12pt]{amsart}
\usepackage{amsmath,amssymb}

\theoremstyle{plain}
\newtheorem{thm}{Theorem}[section]
\newtheorem{theorem}[thm]{Theorem}

\newtheorem{lemma}[thm]{Lemma}
\newtheorem{corollary}[thm]{Corollary}
\newtheorem{proposition}[thm]{Proposition}
\theoremstyle{definition}

\newtheorem{definition}[thm]{Definition}

\numberwithin{equation}{section}

\newcommand{\sC}{{\mathcal C}}

\newcommand{\sK}{{\mathcal K}}

\newcommand{\sU}{{\mathcal U}}

\newcommand{\sZ}{{\mathcal Z}}

\newcommand{\C}{{\mathbb C}}

\newcommand{\BP}{{\mathbb P}}

\title[Equivariant compactification]{Uniqueness of
equivariant compactifications of $\C^n$ by a Fano
manifold of Picard number 1}
\author{Baohua Fu and Jun-Muk Hwang}
\thanks{Baohua Fu is supported by NSFC(11031008 and 11225106) and the KIAS Scholar Program.  Jun-Muk Hwang is supported by
National Researcher Program 2010-0020413 of NRF and MEST}
\begin{document}
\maketitle  \setcounter{tocdepth}{1}

\begin{abstract}
Let $X$ be an $n$-dimensional Fano manifold of Picard number 1. We study
how many different ways $X$ can compactify the complex vector group $ \C^n$
equivariantly. Hassett and Tschinkel showed that when  $X = \BP^n$ with $n \geq 2$,
 there are many distinct ways that $X$ can be realized as
equivariant compactifications of $\C^n$. Our result says that
projective space is an exception: among Fano manifolds of Picard
number 1 with smooth VMRT, projective space is the only one compactifying $\C^n$
equivariantly in more than one ways. This answers questions raised
by Hassett-Tschinkel and Arzhantsev-Sharoyko.
\end{abstract}

\section{Introduction}
Throughout, we will work over complex numbers. Unless stated
otherwise, an open subset is in classical topology.

\begin{definition}\label{d.SEC} Let $G =\C^n$ be the complex vector group
of dimension $n$.   An {\em equivariant
compactification}
 of $G$  is    an algebraic $G$-action
 $A: G \times X \to X$ on a projective variety $X$ of dimension $n$
  with a Zariski open orbit $O \subset X$.   In particular, the orbit $O$
  is $G$-equivariantly
  biregular to $G$. Given a projective variety $X$,  such an action
 $A$ is called an {\em EC-structure} on $X$, in abbreviation of
 `Equivariant Compactification-structure'.
 Let  $A_1: G \times X_1 \to X_1$ and
  $A_2: G \times X_2 \to X_2$ be  EC-structures on two projective varieties $X_1$ and $X_2$. We say
  that $A_1$ and $A_2$
are {\em isomorphic} if there exist a linear automorphism $F:G \to G$ and a biregular morphism $\iota:
X_1 \to X_2$ with the commuting diagram $$
\begin{array}{ccc}
G \times X_1 & \stackrel{A_1}{\longrightarrow} & X_1 \\
(F, \iota) \downarrow & & \downarrow \iota \\
G \times X_2 & \stackrel{A_2}{\longrightarrow} & X_2.
\end{array}$$   \end{definition}

In \cite{HT},  Hassett and Tschinkel studied EC-structures on projective space
$X= \BP^n$.  They proved a correspondence between  EC-structures on $\BP^n$ and Artin local algebras of
 length $n+1$. By classifying  such algebras, they
discovered that there are many distinct isomorphism classes of
EC-structures on $\BP^n$ if $n \geq 2$ and  infinitely many of them if $n \geq 6$.
  They posed the question whether a similar phenomenon
occurs when $X$ is a smooth quadric hypersurface. This was
answered negatively in \cite{S}, using arguments along the line of
Hassett-Tschinkel's approach. A further study was made in
\cite{AS} where the authors raised the corresponding question when
$X$ is a Grassmannian. Even for  simplest examples like
the Grassmannian of lines on $\BP^4$, a direct generalization of the arguments in
\cite{HT} or \cite{S}  seems hard.

Our goal is to give a uniform conceptual answer to these questions,
as follows.

\begin{theorem}\label{t.HT} Let $X$ be a Fano manifold of dimension $n$ with Picard number 1,
 different from $\BP^n$. Assume that $X$ has a family of minimal
rational curves  whose VMRT $\sC_x \subset \BP T_x(X)$ at a
general point $x \in X$ is smooth.  Then all EC-structures
 on $X$ are isomorphic.
\end{theorem}

The meaning of VMRT will be explained in the next section.
A direct corollary is the following.

\begin{corollary}\label{t.c}
Let $X \subset \BP^N$ be a projective submanifold of Picard number
1 such that for a general point $x \in X$, there exists a line of
$\BP^N$ passing through $x$ and lying on $X$.  If $X$
is different from the projective space, then all EC-structures on
 $X$ are isomorphic. \end{corollary}

It is well-known that when $X$ has a projective embedding
satisfying the assumption of Corollary \ref{t.c}, some family of
lines lying on $X$ gives a family of minimal rational curves, for
which the VMRT $\sC_x$ at a general point $x \in X$ is smooth
(e.g. by Proposition 1.5 of \cite{Hw01}). Thus Corollary \ref{t.c}
follows from Theorem \ref{t.HT}.   Corollary \ref{t.c} answers
Arzhantsev-Sharoyko's question on Grassmannians and also gives a
more conceptual answer to Hassett-Tschinkel's question on a smooth
quadric hypersurface, as a Grassmannian or a smooth hyperquadric
can be embedded into projective space with the required property.
In fact,  all examples of Fano manifolds of Picard number 1 known
to the authors, which admit EC-structures, can be embedded into
projective space with the property described in Corollary
\ref{t.c}. These include all irreducible Hermitian symmetric
spaces and some non-homogeneous examples coming from Proposition
6.14 of \cite{FH}.

 Once it is correctly formulated,
the proof of Theorem \ref{t.HT} is a simple consequence of the
Cartan-Fubini type extension theorem in \cite{HM01}, as we will
explain in the next section.

It is natural to ask whether an analogue of Theorem \ref{t.HT} holds for equivariant compactifications
of other linear algebraic groups. Part of our argument can be generalized easily.
 We need, however, an essential
feature of the vector group in Proposition \ref{p.invariant}: any linear automorphism of the tangent
space at the origin comes from an automorphism of the group. New ideas are needed to generalize this part
of the argument to other groups.

\section{Proof of Theorem \ref{t.HT}}

Let us recall the definition of VMRT (see \cite{Hw01} for
details).
 Let $X$ be a Fano manifold. For an irreducible component $\sK$
 of the normalized space ${\rm RatCurves}^n(X)$ of  rational curves on $X$ (see \cite{Kol} for definition),
  denote by $\sK_x \subset \sK$ the subscheme
 parametrizing members of $\sK$ passing through $x \in X$. We say that $\sK$ is a
 family of minimal rational curves if $\sK_x$ is nonempty and projective for a general point
 $x \in X$. By a celebrated result of Mori, any Fano manifold has a family of minimal rational curves.
 Fix a family $\sK$ of minimal rational curves. Denote by $\rho: \sU \to \sK$ and
 $\mu: \sU \to X$ the universal family. Then $\rho$ is a $\BP^1$-bundle and for a general point
 $x \in X$, the fiber $\mu^{-1}(x)$ is complete. By associating a smooth point $x$ of a curve in $X$
 to its tangent direction in $\BP T_x(X)$, we have a rational map $\tau: \sU \dasharrow \BP T(X)$. The proper image of $\tau$ is denoted by $\sC \subset \BP T(X)$ and called the total space of VMRT.  For a general
 point $x \in X$, the fiber $\sC_x \subset \BP T_x(X)$ is called the VMRT of $\sK$ at $x$. In other words,
  $\sC_x$ is the closure
 of the union of tangent directions to members of $\sK_x$ smooth at $x$.

 The concept of VMRT is useful to us via the Cartan-Fubini type extension theorem, Theorem 1.2 in \cite{HM01}.
We will quote a simpler version, Theorem 6.8 of \cite{FH}, where
the condition $\sC_x \neq \BP T_x(X)$ is used instead of our
condition `$X_1, X_2$ different from projective space'. It is
well-known that these two conditions are equivalent from
\cite{CMS}.

 \begin{theorem}\label{t.CF}
 Let $X_1, X_2$ be two Fano manifolds of Picard number 1, different from projective spaces. Let $\sK_1$ (resp. $\sK_2$) be a family of minimal
 rational curves on $X_1$ (resp. $X_2$). Assume that for a general point $x \in X_1$, the VMRT $\sC_x \subset
 \BP T_x(X_1)$ is irreducible and smooth. Let $U_1 \subset X_1$ and $U_2 \subset X_2$ be connected open subsets.
 Suppose there exists a biholomorphic map $\varphi: U_1 \to U_2$ such that for a general point $x \in U_1$, the
 differential $d\varphi_x: \BP T_x(U_1) \to \BP T_{\varphi(x)}(U_2)$ sends $\sC_x$ isomorphically to $\sC_{\varphi(x)}$. Then there exists a biregular morphism $\Phi: X_1 \to X_2$ such that $\varphi= \Phi|_{U_1}.$ \end{theorem}

 The irreducibility of $\sC_x$ is automatic
 in our case. We prove a slightly more general statement.

\begin{proposition}\label{p.irreducible}
Let $X$ be a Fano manifold. Let $H$ be a connected closed subgroup of a connected algebraic group $G$.
Suppose that there exists a $G$-action on $X$ with an open orbit $O\subset X$ such that $H$ is the isotropy
subgroup of a point $o \in O$. Then for any
family of minimal rational curves on $X$, the VMRT $\sC_x$ at  $x \in O$
is irreducible. \end{proposition}

\begin{proof}
 For a family $\sK$ of minimal rational
 curves on $X$, let $\rho: \sU \to \sK$ and $\mu: \sU \to X$ be the universal family.
 Since  $G$ is connected, we have an induced $G$-action on $\sK$ and $\sU$ such that
 $\rho$ and $\mu$ are $G$-equivariant.  For $\alpha \in \mu^{-1}(x)$, denote by $G\cdot
\alpha$ the $G$-orbit of $\alpha$ and by $\mu_{\alpha}: G\cdot \alpha \to O$ the restriction
of $\mu$ to the orbit. Then $\mu_{\alpha}$ is surjective.
Note  that $\mu_{\alpha}^{-1}(x) = H \cdot \alpha.$ Since $H$ is connected,
we see that $\mu_{\alpha}$ has irreducible fibers.

Now to prove the irreducibility of $\sC_x$, it suffices to prove the irreducibility of $\sK_x$.
This is equivalent to showing that the fibers of $\mu: \sU \to X$ are connected.
Suppose not. Let
$$\sU \ \stackrel{\eta}{\longrightarrow} \ \sU' \ \stackrel{\nu}{\longrightarrow} \ X$$ be the Stein factorization, i.e.,   $\eta$ has connected fibers and
$\nu$ is a finite surjective morphism. By our assumption, $\nu$ is not birational.  The image
$\eta (G \cdot \alpha) \subset \sU'
\to X$ is a constructible subset of dimension equal to $\sU'$. Thus it contains a
Zariski dense open subset of $\sU'$.  Since $\mu_{\alpha}: G \cdot \alpha \to O$
 has irreducible fibers, we see that $\nu|_{\eta (G \cdot \alpha)}$ is birational over $O$.
 It follows that $\nu$ is birational, a contradiction. \end{proof}

Although the next lemma is straight-forward, this property of the vector group is essential for our
proof of Theorem \ref{t.HT}.

 \begin{lemma}\label{l}
Let $G$ be the vector  group and let $f$ be a linear automorphism of the tangent space $T_o(G)$ at   the identity $o \in G$. Then  there exists a group automorphism $F: G \to G$  such that $dF_o = f$ as endomorphisms of $T_o(G)$. In particular, denoting by  $r_h: G\to G$ the multiplication (=translation) by $h \in G,$ we have $dF_h\circ d r_h = d r_{F(h)} \circ d F_o$ as homomorphisms from $T_o(G) \mbox{ to } T_{F(h)}(G).$
   \end{lemma}

\begin{proof}
The group automorphism $F$ is just the linear automorphism $f$ viewed as an automorphism of $G$ via the `exponential map' $T_o(G) \cong G$.  We can rewrite the group automorphism property as  $F \circ r_h (g)= r_{F(h)} \circ F (g)$ for any $g,h \in G.$   By taking differentials on both sides, we get $dF_{h} \circ d r_h = d r_{F(h)} \circ d F_o.$ \end{proof}

\begin{proposition}\label{p.invariant}
Let $G$ be the vector group.
Let $X_1, X_2$ be two  projective manifolds equipped with  EC-structures
  $A_i: G \times X_i \to X_i, i=1,2.$
Fix a general point $x_i \in X_i$ such that the orbit $O_i:= A_i(
G, x_i)$  is Zariski open in $X_i$. Let $a_i: G\to O_i$ be the
biregular morphism given by $A_i(\cdot, x_i)$ and let $a^{-1}_{i}:
O_i \to G$  be its inverse. Suppose we have $\sZ_i \subset \BP
T(X_i), i=1,2,$ closed subvarieties dominant over $X_i$ such that
\begin{itemize} \item[(1)] $\sZ_i$ is invariant under the
$G$-actions $A_i$, i.e, for any $g \in G$,
the differential $dg_{x_i}:  \BP T_{x_i}(X_i) \to  \BP T_{g \cdot
x_i}(X_i)$ sends $\sZ_i \cap \BP T_{x_i} (X_i)$ isomorphically to
$\sZ_i \cap \BP T_{g \cdot x_i}(X_i)$; and \item[(2)]  the two
projective varieties $\sZ_1 \cap \BP T_{x_1}(X_1) $ and $\sZ_2
\cap \BP T_{x_2}(X_2)$ are projectively isomorphic. \end{itemize}
Then there exists a group automorphism $F$ of $G$
  such that the biholomorphic map
  $\varphi: O_1 \to O_2$ defined by $\varphi= a_2 \circ F \circ a_1^{-1}$ satisfies
 \begin{itemize} \item[(i)] $\varphi(x_1) = x_2$; \item[(ii)]  $\varphi(g \cdot x_1) =
 F(g) \cdot \varphi(x_1)$ for any $g\in G$; and \item[(iii)] the differential $$d \varphi_{u}: \BP T_{u}(X_1) \to \BP
T_{\varphi(u)}(X_2)$$ sends $\sZ_1 \cap \BP T_u(X_1)$ isomorphically
to $\sZ_2 \cap \BP T_{\varphi(u)}(X_2)$ for all $u \in O_1$.\end{itemize}
\end{proposition}

\begin{proof}
To simplify the notation, let us write $Z_{x}:= \sZ_i \cap \BP T_{x}(X_i)$ when $x \in X_i$.
 Denote by $d a_{i}: T(G) \to T(O_i)$ the
isomorphism of vector bundles given  by the differential of $a_{i}$. By the
assumption (2),   there exists a linear automorphism $f \in {\bf
GL}(T_o(G))$ such that $$ d a_{2} \circ f \circ
d a_{1}^{-1} : T_{x_1}(X_1) \to T_{x_2}(X_2)$$ sends $Z_{x_1}$ isomorphically to  $Z_{x_2}$, i.e.
 $$ f (d a_{1}^{-1}(Z_{x_1})) = d a_{2}^{-1}(Z_{x_2}).$$
Let $F$ be the group automorphism of $G$ induced by $f$ in Lemma \ref{l}.
 It is clear that $\varphi = a_2 \circ F \circ a_1^{-1}$ is a  biregular map satisfying (i).
From $G$-equivariance of $a_1$ and $a_2$,
 $$\varphi (g \cdot x_1) = a_2 \circ F \circ a_1^{-1} (g \cdot x_1) =  a_2 \circ F(g \cdot o) = a_2 (F(g) \cdot o)
= F(g) \cdot a_2(o) = F(g) \cdot x_2.$$ This proves (ii).
  For any $u \in O_1$, let $h= a_1^{-1}(u)$ such that $u = h \cdot x_1$. Then using Lemma \ref{l} and the condition (2),
  \begin{eqnarray*} d \varphi_{u} (Z_u) & = & d a_2 \circ d F \circ d a_1^{-1}(Z_{h\cdot x_1})  \\ &=&
  d a_2 \circ d F \circ d r_h \circ d a_1^{-1}(Z_{x_1}) \\ &=& d a_2 \circ d r_{F(h)} \circ d F \circ d a_1^{-1}(Z_{x_1}) \\ &=&   d a_2 \circ d r_{F(h)} \circ f ( a_1^{-1}(Z_{x_1})) \\   &=& d a_2 \circ d r_{F(h)} \circ d a_2^{-1}(Z_{x_2}) \\ & = & Z_{F(h) \cdot x_2} = Z_{\varphi (u)}. \end{eqnarray*} The last equality is
  from $F(h)\cdot x_2 = a_2 \circ F(h) = a_2 \circ F (a_1^{-1} (u)) = \varphi (u).$
This shows (iii).
  \end{proof}

%

\begin{proof}[Proof of Theorem \ref{t.HT}]
Let $A_i: G \times X \to X, i=1,2,$ be two EC-structures on $X$ with open orbits
$O_i \subset X$.
 Fix a family of minimal rational curves on $X$ with
  the total space of VMRT $\sC \subset \BP T(X)$.
  Set $X_1=X_2=X$ and $\sZ_1=\sZ_2=\sC$.
 The conditions (1) and (2) of Proposition \ref{p.invariant}
  are satisfied. We obtain two open subsets $O_1, O_2 \subset X$ with
 a biholomorphism $\varphi: O_1 \to O_2$ whose differential
 sends $\sC_u$ to $\sC_{\varphi(u)}$ for all $u \in O_1$.
 By Proposition \ref{p.irreducible},
  we can apply Theorem \ref{t.CF} to obtain
 an automorphism $\Phi: X \to X$ satisfying $\Phi|_{O_1} = \varphi.$
  From the property (ii) of $\varphi$ in
  Proposition \ref{p.invariant},
$$\Phi \circ A_1 |_{G \times O_1}= A_2 \circ (F \times \Phi) |_{G \times O_1}.$$
Then the equality $\Phi \circ A_1= A_2 \circ  (F \times \Phi)$ must hold on the whole
$G \times X$. Thus the two EC-structures $A_1$ and $A_2$ are
isomorphic. \end{proof}

\bigskip
Baohua Fu

Institute of Mathematics, AMSS, Chinese Academy of Sciences, 55
ZhongGuanCun

East Road, Beijing, 100190, China and
 Korea
Institute for Advanced Study,

 Hoegiro 87, Seoul, 130-722, Korea

 bhfu@math.ac.cn

\bigskip
Jun-Muk Hwang

 Korea Institute for Advanced Study, Hoegiro 87,

Seoul, 130-722, Korea

jmhwang@kias.re.kr
\end{document}